\documentclass{amsart}

\usepackage[all]{xy}
\usepackage{amsmath}
\usepackage{hyperref}
\usepackage{amsfonts,graphics,amsthm,amsfonts,amscd,latexsym}
\usepackage{epsfig}
\usepackage{flafter}
\usepackage{mathtools}
\usepackage{comment}

\hypersetup{
    colorlinks=true,    
    linkcolor=blue,          
    citecolor=blue,      
    filecolor=blue,      
    urlcolor=blue           
}
\usepackage{tikz}
\usetikzlibrary{graphs,positioning,arrows,shapes.misc,decorations.pathmorphing}

\tikzset{
    >=stealth,
    every picture/.style={thick},
    graphs/every graph/.style={empty nodes},
}

\tikzstyle{vertex}=[
    draw,
    circle,
    fill=black,
    inner sep=1pt,
    minimum width=5pt,
]
\usepackage[position=top]{subfig}
\usepackage[alphabetic,backrefs]{amsrefs}
\usepackage{amssymb}
\usepackage{color}

\calclayout

\usetikzlibrary{decorations.pathmorphing}
\tikzstyle{printersafe}=[decoration={snake,amplitude=0pt}]

\newcommand{\Hom}{\operatorname{Hom}}

\renewcommand{\qq}{\mathbb{Q}}
\newcommand{\zz}{\mathbb{Z}}

\def\O#1.{\mathcal {O}_{#1}}			
\def\pr #1.{\mathbb P^{#1}}				
\def\af #1.{\mathbb A^{#1}}			
\def\ses#1.#2.#3.{0\to #1\to #2\to #3 \to 0}	
\def\xrar#1.{\xrightarrow{#1}}			
\def\K#1.{K_{#1}}						
\def\bA#1.{\mathbf{A}_{#1}}			
\def\bM#1.{\mathbf{M}_{#1}}				
\def\bL#1.{\mathbf{L}_{#1}}				
\def\bB#1.{\mathbf{B}_{#1}}				
\def\bK#1.{\mathbf{K}_{#1}}			
\def\subs#1.{_{#1}}					
\def\sups#1.{^{#1}}

\usepackage{tikz}
\usetikzlibrary{matrix,arrows,decorations.pathmorphing}

  \newtheorem{introthm}{Theorem}

  \newtheorem{introcor}{Corollary}
  \newtheorem{introprop}{Proposition}
  
  \newtheorem{theorem}{Theorem}[section]

  \newtheorem{definition}[theorem]{Definition}

\theoremstyle{remark}

\numberwithin{equation}{section}

\begin{document}

\title[Extracting non-canonical places]{Extracting non-canonical places}

\author[J.~Moraga]{Joaqu\'in Moraga}
\address{Department of Mathematics, Princeton University, Fine Hall, Washington Road, Princeton, NJ 08544-1000, USA
}
\email{jmoraga@princeton.edu}

\maketitle

\begin{abstract}
Let $(X,B)$ be a log canonical pair and $\mathcal{V}$ be a finite set of divisorial valuations with log discrepancy in $[0,1)$.
We prove that there exists a projective birational morphism $\pi \colon Y\rightarrow X$ so that the exceptional divisors are $\qq$-Cartier and correspond to elements of $\mathcal{V}$.
We study how two such models are related.
Moreover, we provide an application to the study of deformations of log canonical singularities.
\end{abstract}

\setcounter{tocdepth}{1} 
\tableofcontents

\section{Introduction}

In this article, we study the extraction of log canonical places from a log canonical pair $(X,B)$.
In particular, we are interested in studying models that extract a prescribed finite set of divisorial valuations with log discrepancy in $[0,1)$.
Our proof is an inductive argument in which we extract such valuations one-by-one. The main techniques involved are the minimal model program for $\qq$-factorial dlt pairs over a birational base~\cite{BCHM10} and the existence of certain good minimal models due to Birkar, Hacon and Xu~\cite{Bir12,HX13}.
We will prove the following result.

\begin{introthm}\label{introthm:extraction}
Let $(X,B)$ be a log canonical pair with $B$ a $\qq$-divisor. Let  $\mathcal{V}$ be a finite set of divisorial valuations with log discrepancy in $[0,1)$.
Then, there exists a projective birational morphism $\pi \colon Y\rightarrow X$ satisfying the following conditions:
\begin{enumerate}
    \item The exceptional locus ${\rm Ex}(\pi)$ is purely divisorial,
    \item the exceptional divisors are $\qq$-Cartier, 
    \item the exceptional divisors correspond to elements of $\mathcal{V}$, and
    \item there exists an anti-effective divisor supported on the exceptional locus which is ample over $X$.
\end{enumerate}
In particular, we may find a boundary divisor $B_Y$ on $Y$ satisfying the following conditions:
\begin{enumerate}
 \setcounter{enumi}{4}
    \item $B_Y$ and $\pi^{-1}_* B$ only differ at exceptional divisors of $\pi$,
    \item $(Y,B_Y)$ is log canonical, 
    \item $K_Y+B_Y$ is ample over $X$, and
    \item no log canonical center of $(Y,B_Y)$ is contained in ${\rm Ex}(\pi)$.
\end{enumerate}
\end{introthm}

It is well-known for the experts that we can extract all the valuations in $\mathcal{V}$ in a $\qq$-factorial dlt modification of $(X,B)$.
However, such a model would often extract more exceptional divisors and contain small components in the exceptional locus.
Controlling the number of exceptional divisors on the dlt modifications seems impossible.
Indeed, the log canonical places extracted by a dlt modification come from log canonical places extracted in a log resolution.
Hence to control the former, we would need to control the construction of log resolutions.
Instead, in this article, we will start from a dlt modification and contract the extra divisors
by running several MMP's over different bases.
Applying the above result, we can generalize the extraction of non-canonical places of klt pairs~\cite[Corollary 1.4.3]{BCHM10}.

\begin{introcor}\label{introcor:extraction}
Let $(X,B)$ be a log canonical pair with $B$ a $\qq$-divisor. Let $\mathcal{V}$ be a finite set of divisorial valuations with log discrepancy in $[0,1]$.
Furthermore, assume that no canonical center contains a log canonical center.
Then, there exists a projective birational morphism $\pi \colon Y\rightarrow X$ so that the exceptional divisors are $\qq$-Cartier and correspond to elements of $\mathcal{V}$.
\end{introcor}

It is natural to ask how to relate two such birational extractions.
Once we fix the boundary $B_Y$ such model is unique 
due to the uniqueness of ample models over the base.
However, it is interesting to understand how two models $Y_1 \rightarrow X$ and $Y_2 \rightarrow X$ satisfying 
$(1)$-$(3)$ are related.
In this direction, we can prove the following statement up to dimension three.

\begin{introprop}\label{introprop:comparison}
Let $(X,B)$ be a log canonical pair of dimension at most three with $B$ a $\qq$-divisor.
Let $\mathcal{V}$ be a finite set of divisorial valuations with log discrepancy in $[0,1)$.
Let $Y_1\rightarrow X$ and $Y_2\rightarrow X$ two birational models satisfying $(1)$-$(3)$ of Theorem~\ref{introthm:extraction}.
Then the small birational map $Y_1\dashrightarrow Y_2$ factors as  a sequence of flops over $X$ and possibly a small contraction over $X$.
In particular, if $\rho(Y_1/X)=\rho(Y_2/X)$, then the small birational map $Y_1\dashrightarrow Y_2$ factors as a sequence of flops over $X$.
\end{introprop}

Here, $\rho(Y/X)$ is the rank of the $\qq$-vector space generated by $\qq$-Cartier divisors on $Y$ modulo $\qq$-linear equivalence over $X$.
In higher dimensions, the above statement follows from the minimal model program for log canonical pairs over a birational morphism and the abundance conjecture.

It is an interesting question whether a birational morphism admits an anti-effective divisor supported in the exceptional locus which is ample over the base.
This is the case for instance if the birational morphism is a composition of blow-ups of centers of codimension at least two.
As an application of the main theorem, 
we may find such divisor on a small model of our birational map
whenever we extract log discrepancies in the interval $[0,1)$.

\begin{introcor}\label{introcor:existence-ample-anti-effective}
Let $(X,B)$ be a log canonical pair and $\pi \colon Y\rightarrow X$ be a birational morphism which only extract divisors with log discrepancy in $[0,1)$. 
Then, there exists a small morphism $Y\dashrightarrow Y'$ over $X$, so that $\pi'\colon Y'\rightarrow X$ admits an anti-effective divisor supported on ${\rm Ex}(\pi')$ which is ample over the base.
\end{introcor}

Another straightforward application is the extraction of a unique  exceptional divisor $E$ over a log canonical germ whose normalization carries a log Calabi-Yau structure.
We remark that it is not expected that $E$ carries the structure of a slc pair, since in general, it may not be $S_2$ (see, e.g.,~\cite{Kol92}).
This an analog of the so-called plt blow-up for klt singularities~\cite{Xu14}.

\begin{introcor}\label{introcor:lcy-extraction}
Let $x\in(X,B)$ be a log canonical singularity, 
then there exists a birational morphism $\pi \colon Y\rightarrow X$ which extracts a unique divisor $E$ mapping onto $x$, so that $-E$ is ample over $X$, and $(Y,B_Y+E)$ is log canonical.
Here, $B_Y$ is the strict transform of $B$ on $Y$.
Moreover, if $E^\eta \rightarrow E$ is the normalization of $E$,
then there exists a boundary $B^\eta$ so that $(E^\eta,B^\eta)$ is a log Calabi-Yau pair.
\end{introcor}

Given a log canonical pair $(X,B)$ which is not klt, 
we can ask how many divisors we need to extract over $X$ 
to obtain a variety with klt singularities.
The following application states that such number equals the number of maximal log canonical centers of $(X,B)$, i.e., 
the number of log canonical centers which are maximal with respect to the inclusion.

\begin{introcor}\label{introcor:turning-klt}
Let $(X,B)$ be a log canonical pair which is not klt.
There exists a projective birational morphism $Y\rightarrow X$
so that $Y$ has klt singularities, 
and $\pi$ extracts exactly one divisor over each maximal log canonical center of $(X,B)$.
\end{introcor}

Finally, we apply the main theorem to study degenerations of log canonical singularities.
The following statement is a generalization of the fact that 
klt singularities deform to cone over Fano type varieties~\cite{LX16}.

\begin{introthm}\label{introthm:degeneration}
Log canonical singularities degenerate to singularities whose normalizations are log canonical cones over log Calabi-Yau pairs.
Moreover, isolated $\qq$-factorial log canonical singularities degenerate to cones over slc Calabi-Yau pairs.
\end{introthm}

\section{Preliminaries}

We work over an algebraically closed field of characteristic zero.

\begin{definition}{\em 
A {\em log pair} is a pair $(X,B)$ consisting of a normal quasi-projective variety $X$ and an effective $\qq$-divisor $B$ so that $K_X+B$ is $\qq$-Cartier.
Given a projective birational morphism $\pi \colon Y \rightarrow X$ from a normal quasi-projective variety $Y$,
and a prime divisor $E$ on $Y$, 
we define the {\em log discrepancy} of $(X,B)$ at $E$ to be
\[
a_E(X,B):=1-{\rm coeff}_E(\pi^*(K_X+B)).
\]
We say that a log discrepancy is exceptional if the center of $E$ on $X$ is not a divisor.
We say that a log pair $(X,B)$ is {\em  terminal} (resp. {\em canonical}) if all its exceptional log discrepancies are greater than one (resp. greater than or equal to one).
We say that a log pair $(X,B)$ is {\em Kawamata log terminal} (resp. {\em log canonical}) if all its log discrepancies are positive (resp. non-negative).
We may use the usual abbreviation lc (resp. klt) for log canonical (resp. Kawamata log terminal).
}
\end{definition}

\begin{definition}{\em 
Let $(X,B)$ be a log canonical pair.
A divisor $E$ over $X$ is said to be a {\em log canonical place} (resp. {\em canonical place}) if $a_E(X,B)=0$ (resp. $a_E(X,B)=1$).
The image of a log canonical place (resp. canonical place)
in $X$ is said to be a {\em log canonical center} (resp. {\em canonical center}).
}
\end{definition}

\begin{definition}
{\em
A log pair $(X,B)$ is said to be {\em divisorially log terminal} (or dlt for short) if there exists an open set $U\subset X$ satisfying the following conditions:
\begin{enumerate}
    \item the coefficients of $B$ are at most one,
    \item $U$ is smooth and $B|_U$ has simple normal crossing, and
    \item any log canonical center of $(X,B)$ intersect $U$ and is given by strata of $\lfloor B \rfloor$.
\end{enumerate}
Given a log canonical pair $(X,B)$, we say that $\pi \colon Y\rightarrow X$ is a {\em $\qq$-factorial dlt modification}
if $Y$ is $\qq$-factorial, $\pi$ only extract log canonical places, and $\pi^*(K_X+B)=K_Y+B_Y$ is a dlt pair.
The existence of $\qq$-factorial dlt modifications for log canonical pairs is well-known (see, e.g.~\cite{KK10}).
}
\end{definition}

\begin{definition}
{\em 
A demi-normal scheme is a $S_2$ scheme whose codimension one points are either regular or nodes.
Let $\pi\colon X^\eta \rightarrow X$ be the normalization morphism.
The conductor ideal $\Hom_X(\pi_*\mathcal{O}_{X^\eta},\mathcal{O}_X)$ is the largest ideal sheaf on $X$ which is also an ideal sheaf on $X^\eta$.
Therefore, it defines two subschemes $D$ and $D^\eta$ on $X$ and $X^\eta$ respectively which we call the {\em conductor subschemes}.

Let $X$ be a demi-normal scheme.
Denote by $D$ its conductor.
Let $B$ be an effective divisor on $X$ whose support does not contain any component of the conductor $D$.
Let $\pi \colon X^\eta \rightarrow X$ be the normalization morphism and $B^\eta$ the divisorial part of $\pi^{-1}(B)$.
We say that $(X,B)$ is {\em semi log canonical} if $K_X+B$ is $\qq$-Cartier and $K_{X^\eta}+B^\eta+D^\eta$ is log canonical.
}
\end{definition}

\section{Proof of the main Theorem}

\begin{proof}[Proof of Theorem~\ref{introthm:extraction}]
We will construct the morphism $Y\rightarrow X$ inductively 
by extracting the non-canonical places one-by-one.
We denote by $\mathcal{V}=\{v_1,\dots,v_k\}$
the finite set of divisorial valuations.

Let $Y_1 \rightarrow X$ be a $\qq$-factorial dlt modification of $(X,B)$.
We may assume that all the divisorial valuations of $\mathcal{V}$ which have log discrepancy zero are indeed divisors on $Y_1$.
Since $Y_1$ is klt and $\qq$-factorial, we may apply~\cite[Corollary 1.4.3]{BCHM10} to further extract all divisorial valuations of $\mathcal{V}$ with positive log discrepancy.
Hence, we may assume $Y_1\rightarrow X$ extracts all the divisorial valuations of $\mathcal{V}$.
Moreover, we have that the log pull-back of $K_X+B$ to $Y_1$ is $\qq$-factorial and log canonical.
We further replace $Y_1$ with a $\qq$-factorial dlt modification for the log pull-back of $K_X+B$ to $Y_1$.
Let $\pi_1 \colon Y_1\rightarrow X$ be the projective birational morphism.
We can write 
\[
\pi_1^*(K_X+B)=K_{Y_1}+B_{Y_1}+E_{Y_1} \sim_{\qq,X} 0,
\]
where $B_{Y_1}$ is the strict transform of $B$ on $Y_1$, and 
$E_{Y_1}$ is an effective divisor supported on the exceptional divisors.
By construction, the above pair is $\qq$-factorial and dlt.
Let $F_{Y_1}$ be the center of $v_1$ on $Y_1$.
We may run a minimal model program for
\[
K_{Y_1}+B_{Y_1}+E_{Y_1} -\epsilon_1 F_{Y_1} \sim_{\qq, X} -\epsilon_1 F_{Y_1}
\]
over $X$ with scaling of an ample divisor.
The existence of such a minimal model program follows from~\cite{BCHM10}.
By~\cite[Theorem 1.6]{HX13}, this minimal model program terminates with a $\qq$-factorial good minimal model $Y_1'$.
Let $K_{Y'_1}+B_{Y'_1}+E_{Y'_1}-\epsilon_1 F_{Y'_1}$ be the relatively semiample divisor over $X$.

We claim that the MMP $Y\dashrightarrow Y_1'$ over $X$
does not contract $F_{Y_1}$.
We proceed by contradiction.
Assume that the minimal model program $Y\dashrightarrow Y_1'$ over $X$ contract $F_{Y_1}$.
By monotonicity of log discrepancies under the minimal model program, we obtain
\[
\alpha:=a_{v_1}(X,B)=a_{v_1}(K_{Y'_1}+B_{Y'_1}+E_{Y'_1}) >
a_{v_1}(K_{Y_1}+B_{Y_1}+E_{Y_1} -\epsilon_1 E_{Y_1})= \alpha+\epsilon_1,
\]
leading to a contradiction.
Hence, the center of $F_1$ on $Y'_1$ is a divisor.

Let $X_1\rightarrow X$ be the relative ample model for
\[
K_{Y_1}+B_{Y_1}+E_{Y_1} -\epsilon_1 F_{Y_1}
\]
over $X$. 
By construction, the morphism $Y_1'\rightarrow X_1$ is $F_{Y_1}$-trivial.
Hence, if $Y_1'\rightarrow X_1$ contract $F_{Y'_1}$,
then we obtain 
\[
\alpha =a_{E_1}(K_{X}+B) = a_{E_1}(K_{Y_1}+B_{Y_1}+E_{Y_1}-\epsilon_1 F_{Y_1}) =\alpha+\epsilon_1
\]
leading to a contradiction.
Thus, we conclude that the strict transform of $F_{Y_1'}$ on $X_1$ is a divisor.
We call such divisor $F_{X_1}$.
By construction, $F_{X_1}$ is $\qq$-Cartier
and $-F_{X_1}$ is ample over $X$.
In particular, the exceptional locus of $\phi_1 \colon X_1 \rightarrow X$ is purely divisorial
and extract a unique divisor $F_{X_1}$.
Hence, we have that 
\[
K_{X_1}+B_{X_1}+(1-\alpha) F_{X_1} \sim_{\qq,X} 0,
\]
is log canonical, and 
\[
K_{X_1}+B_{X_1}+(1-\alpha-\epsilon_1)F_{X_1} 
\]
is log canonical and ample over $X$.
It is clear that all the log canonical centers
of $(X_1,B_{X_1}+(1-\alpha) F_{X_1})$
are either the strict transform of a log canonical center
of $(X,B)$ or are contained in the support of $F_{X_1}$.
Hence, the log canonical centers of
$(X_1,B_{X_1}+(1-\alpha-\epsilon_1)F_{X_1})$ 
are the strict transform of the log canonical centers of $(X,B)$.
We define $B_1:=B_{X_1}+(1-\alpha-\epsilon_1)F_{X_1}$.
Thus, the projective birational morphism $X_1\rightarrow X$ 
and the pair $(X_1,B_1)$ satisfy the conditions of the statement 
for $\mathcal{V}_1:=\{v_1\}$.

Proceeding inductively, we can find a sequence of birational extractions 
\[
X =: X_0 \leftarrow X_1\leftarrow X_2\leftarrow \dots\leftarrow X_k,
\]
where each $X_i\rightarrow X_{i-1}$, with $i\in \{1,\dots, k\}$, only extracts the divisor corresponding to $v_i$.
Hence, we may define $Y:=X_k$.
Observe that the centers of $v_1,\dots,v_k$ on $Y$ are $\qq$-Cartier.
We denote such centers by $F_i$.
Hence, the morphism $Y\rightarrow X$ satisfies the conditions
$(1)$-$(3)$ of the statement.
It suffices to prove that $Y\rightarrow X$ admits an anti-effective divisor supported on the exceptional locus which is ample over the base.
By construction, we know that for
each $i\in \{1,\dots,k\}$
the divisor $-F_{X_i}$ (the center of $v_i$ on $X_i$ is ample over $X_{i-1}$.
Then, we can choose
\[
1\gg \epsilon_1 \gg \epsilon_2 \gg \dots \gg \epsilon_k >0,
\]
so that the divisor
\[
-\sum_{i=1}^k \epsilon_i F_i
\]
has support equal to ${\rm Ex}(\pi)$ and is ample over $X$.
Let $(Y,\Delta_Y)$ be the log pull-back of $(X,B)$.
We define the divisor
\[
B_Y:=\Delta_Y - \sum_{i=1}^k \epsilon_i F_i.
\]
Thus, the pair $K_Y+B_Y$ is log canonical 
and ample over $X$.
Moreover, no log canonical center of $K_Y+B_Y$ is contained in the exceptional locus of $Y\rightarrow X$.
In particular, all the log canonical centers of $(Y,B_Y)$ are strict transforms of the log canonical centers of $(X,B)$
which are not contained in the image of the exceptional locus.
\end{proof}

\begin{proof}[Proof of Corollary~\ref{introcor:extraction}]
Since no canonical center of $(X,B)$ contains a log canonical center of $(X,B)$, we can add an effective ample divisor $A$ which contains all the canonical centers and doesn't contain any log canonical center.
In particular, $(X,B+A)$ is log canonical, and all the canonical centers of $(X,B)$ are non-canonical centers of $(X,B+A)$.
Hence, the finite set $\mathcal{V}$ of divisorial valuations with log discrepancy in $[0,1]$ for $(X,B)$
becomes a set of divisorial valuations with log discrepancy
in $[0,1)$ for $(X,B+A)$.
Fianlly, we can apply Theorem~\ref{introthm:extraction} to the log pair $(X,B+A)$ to extract all the valuations corresponding to elements of $\mathcal{V}$.
\end{proof}

\begin{proof}[Proof of Propostion~\ref{introprop:comparison}]
By Theorem~\ref{introthm:extraction}, we know that there exists a divisor $B_{Y_2}$ on $Y_2$ so that $(Y_2,B_{Y_2})$ is an ample log canonical model over $X$.
Moreover, the log pull-back of $(X,B)$ to $Y_2$ and $K_{Y_2}+B_{Y_2}$ differ only at the exceptional divisors of $Y_2\rightarrow X$.
Hence, if we let $B_{Y_1}$ to be the strict transform of $B_{Y_2}$ on $Y_1$, we have that $(Y_1,B_{Y_1})$ is a log canonical pair.
We can run a minimal model program for $K_{Y_1}+B_{Y_1}$ with scaling of an ample divisor over $X$.
Such minimal model program contract no divisors, so it is a sequence of flops relative to $X$ for the log pull-back of $(X,B)$.
It terminates with a good minimal model over $X$
whose ample model is $(Y_2,B_{Y_2})$.
The morphism to the ample model is a small contraction to $Y_2$ over $X$.
Thus, we achieved to factor the morphism $Y_1\dashrightarrow Y_2$
as a sequence of flops over $X$ and possibly a small contraction over $X$.
If $\rho(Y_1/X)=\rho(Y_2/X)$, then the minimal model and the ample model agree, so there is no such small contraction.
\end{proof}

\begin{proof}[Proof of Corollary~\ref{introcor:existence-ample-anti-effective}]
Let $\mathcal{V}$ be the set of divisorial valuations extracted by $\pi$.
Applying Theorem~\ref{introthm:extraction} to the set of valuations $\mathcal{V}$ of the log canonical pair $(X,B)$,
we get such small projective birational morphism $\pi'\colon Y'\rightarrow X$.
\end{proof}

\begin{proof}[Proof of Corollary~\ref{introcor:lcy-extraction}]
Let $x\in (X,B)$ be a log canonical singularity.
Without loss of assumptions, we may assume that $x$
is a log canonical center.
Otherwise, we may cut down the minimal log canonical center containing $x$ by adding a $\qq$-Cartier ample divisor.
Hence, there exists a divisorial valuation $v$ whose center on $X$ is $x$.
By Theorem~\ref{introthm:extraction} applied to the set $\mathcal{V}=\{ v\}$, we conclude that there exists a projective birational morphism $\pi \colon Y\rightarrow X$ that extracts a unique divisor $E$ over $X$ which is anti-ample over $X$, and $(Y,B_Y+E)$ is log canonical.
Here, we are denoting by $B_Y$ the strict transform of $B$ on $Y$.
By construction $K_Y+B_Y+E\sim_{\qq,X} 0$, hence the restriction
$K_Y+B_Y+E|_E$ is $\qq$-linearly trivial.
The last statement follows from the theory of adjunction to log canonical places~\cite{Hac14}.
\end{proof}

\begin{proof}[Proof of Corollary~\ref{introcor:turning-klt}]
For each maximal log canonical center $Z_i$ of $(X,B)$ we may find a log canonical place that maps to it.
We denote by $v_i$ the corresponding divisorial valuation.
Since $(X,B)$ has finitely many log canonical centers, then
we may find a finite set of divisorial valuations $v_1,\dots,v_k$ each of them corresponding to a single maximal log canonical center of $(X,B)$.
Let $\pi \colon Y \rightarrow X$ be the projective birational morphism constructed by Theorem~\ref{introthm:extraction}
applied to the finite set $\mathcal{V}$.
Then, by construction $(Y,B_Y)$ has klt singularities.
\end{proof}

\begin{proof}[Proof of Theorem~\ref{introthm:degeneration}]
By Corollary~\ref{introcor:lcy-extraction} there exists a projective birational morphism $\pi \colon Y \rightarrow X$ which extracts a unique divisor $E$ mapping onto $x$, so that
$-E$ is ample over $X$, and $(Y,B_Y+E)$ is log canonical.
Moreover, if $E^\eta \rightarrow E$ is the normalization of $E$,
then there exists a boundary $B^\eta$ so that $(E^\eta,B^\eta)$ is a log Calabi-Yau pair.
We denote by $v={\rm ord}_E$.
Consider the graded ring 
\[A:=\bigoplus_{k=0}^\infty a_k(v)/a_{k-1}(v),
\]
where 
\[
a_k(v):=\{ f\in \mathbb{K}(X) \mid v(f)\geq k\}
\]
is an ideal sheaf on $X$.
We can also consider the extended Rees algebra of the above ring
\[
\mathcal{R}=\bigoplus_{k\in \zz} a_k(v)t^{-k}\subset A[t,t^{-1}].
\]
The affine variety $\mathcal{X}:={\rm Spec}(\mathcal{R})$ admits a flat morphism to $\mathbb{A}^1$ and its general fiber is isomorphic to $X$.
Moreover, the central fiber is isomorphic to $X_0:={\rm Spec}(A)$.
We claim that ${\rm Spec}(A)$ is isomorphic to the orbifold cone over $E$
with respect to the $\qq$-polarization $-E|_E$.
Since $E$ may be non-normal, the orbifold cone may be non-normal as well.
Let $c$ be a natural number so that $cE$ is Cartier on $Y$. By the restriction exact sequence, we have
\[
0\rightarrow \mathcal{O}_{Y}(-(c+1)E)\rightarrow \mathcal{O}_{Y}(-cE)\rightarrow \mathcal{O}_{E}(-cE)\rightarrow 0,
\]
Observe that $-(c+1)E-K_Y-B_Y\sim_\qq -cE$
which is ample over $X$ and $(Y,B_Y)$ has log canonical singularities.
Hence, we may apply a relative version of Kawamata-Viehweg vanishing for log canonical pairs~\cite[Theorem 1.7]{Fuj14}, to deduce that
\[
R^1\pi_*\mathcal{O}_X(-(c+1)E))\simeq
H^1(Y,\mathcal{O}_Y(-(c+1)E))=0.
\]
So, we have an isomorphism
\[
H^0(E,\mathcal{O}(-cE|_{E})) \simeq H^0(\mathcal{O}_{Y}(-cE))/H^0(\mathcal{O}_{Y}(-(c+1)E)) \simeq a_{c}(v)/a_{c+1}(v).
\]
Then $A^{(c)}=\bigoplus_{k=0}^\infty a_{ck}(v)/a_{ck+1}(v)$ is isomorphic to the cone
over $E$ with respect to the polarization $-cE|_{E}$. 
We define
\[
X_0^{(c)}:={\rm Spec}(A^{(c)}).
\]
Note that we have a cyclic $c$-th cover $X_0^{(c)}\rightarrow X_0$.
This cyclic cover ramifies over the points of $E$ on which $E$ is not a Cartier divisor on $Y$.
This proves the claim.
On the other hand, the normalization of $X_0$ is the cone over 
$E^\nu$ with respect to the $\qq$-polarization given by the pull-back of $-E|_E$ to $E^\nu$.
Since $(E^\nu,B^\nu)$ is a log canonical log Calabi-Yau pair, 
we conclude that the cone is a log canonical cone 
over a log Calabi-Yau pair~\cite{Kol13}.

Finally, if $x\in (X,B)$ is a $\qq$-factorial isolated singularity,
we conclude that $(Y,B_Y)$ is a $\qq$-factorial klt pair.
In particular, $Y$ is Cohen-Macaulay, so $E$ is Cohen-Macaulay as well.
We deduce that $E$ is a $S_2$ scheme, so it is semi-log canonical.
Hence, the central fiber of the above degeneration is a cone over a semi-log canonical log Calabi-Yau pair.
\end{proof}

\end{document}